\documentclass[11pt]{amsart}

\usepackage[a4paper,margin=30mm]{geometry}
\usepackage[T1]{fontenc}
\usepackage[utf8]{inputenc}
\usepackage{lmodern}
\usepackage{microtype}
\usepackage{amsmath,amssymb,amsthm,mathtools}
\usepackage{booktabs,tabularx,array}
\usepackage{enumitem}
\usepackage{xcolor}
\usepackage{listings}
\usepackage{tikz-cd}
\usepackage{hyperref}
\usepackage[nameinlink,capitalise,noabbrev]{cleveref}
\usepackage{comment}

\hypersetup{
  colorlinks=true,
  linkcolor=blue!45!black,
  citecolor=green!35!black,
  urlcolor=blue!55!black,
  pdftitle={The Real Quotient of the Cartwright--Steger Surface},
  pdfauthor={Andras Stipsicz and Zoltan Szabo}
}

\setlist{topsep=4pt,itemsep=2pt,parsep=1pt}

\definecolor{codebg}{RGB}{247,248,250}
\definecolor{codeframe}{RGB}{205,210,218}
\newtheorem{theorem}{Theorem}[section]
\newtheorem{proposition}[theorem]{Proposition}
\newtheorem{lemma}[theorem]{Lemma}
\newtheorem{corollary}[theorem]{Corollary}
\theoremstyle{definition}
\newtheorem{definition}[theorem]{Definition}
\theoremstyle{remark}
\newtheorem{remark}[theorem]{Remark}

\newcommand{\Z}{\mathbb Z}
\newcommand{\Q}{\mathbb Q}
\newcommand{\R}{\mathbb R}
\newcommand{\C}{\mathbb C}
\newcommand{\CP}{\mathbb{CP}}
\newcommand{\RP}{\mathbb{RP}}
\newcommand{\B}{\mathbb B^2_{\C}}
\newcommand{\Fix}{\operatorname{Fix}}

\newcommand{\sgn}{\operatorname{sign}}
\newcommand{\GammaMax}{\overline\Gamma}
\newcommand{\PiCS}{\Pi}
\newcommand{\Alb}{\operatorname{Alb}}
\newcommand{\Aut}{\operatorname{Aut}}
\newcommand{\id}{\operatorname{id}}
\newcommand{\normal}[1]{\mathopen{\langle\!\langle}#1\mathclose{\rangle\!\rangle}}
\newcommand{\zetaT}{\zeta_{12}}
\newcommand{\omegaE}{\omega}
\newcommand{\phiA}{\varphi}

\title{The real quotient of the Cartwright--Steger surface}
\author{Andr\'as I. Stipsicz}
\author{Zolt\'an Szab\'o}
\date{August 2026}

\begin{document}

\begin{abstract}
Borisov and Yeung showed that the Cartwright--Steger surface $X$ is
defined over $\mathbb Q$, hence complex conjugation gives an
antiholomorphic involution $c_{\mathrm{BY}}$ on $X$.
We show that the orbit space
$Y=X/\langle c_{\mathrm{BY}}\rangle$ is a closed oriented smooth four-manifold
homeomorphic to $\CP^2\#(S^1\times S^3)$, presenting $X$ as a double
branched cover $X\to Y$, with branch locus diffeomorphic to $\#_3 {\RP}^2$.
\end{abstract}

\maketitle

\section{Introduction}

The classification of fake projective planes (compact complex surfaces
of general type with the same rational cohomology ring as the complex
projective plane $\CP^2$) by Prasad--Yeung and Cartwright--Steger
\cite{PY,CS} led to the discovery of another interesting complex
surface, the \emph{Cartwright--Steger surface} $X$~\cite{CS}. It is a compact
complex surface with Euler characteristic $\chi (X)=3$, signature
$\sgn (X)=1$ and first Betti number $b_1(X)=2$. It follows that $X$ is
on the Bogomolov--Miyaoka--Yau line and (like all fake projective
planes) it is a quotient of the complex two-ball \cite{Miyaoka,Yau}.
It is known that a fake projective plane admits no antiholomorphic
involution \cite[Theorem~5.1]{KharKul}, while complex conjugation on $\CP ^2$
provides such a map.  The quotient of $\CP ^2$ by this involution was
shown to be diffeomorphic to $S^4$, while the fixed point set of this
involution in $\CP ^2$ is
diffeomorphic to ${\RP} ^2$ \cite{Kuiper, Massey}.

Cartwright, Koziarz and Yeung \cite{CKY} gave an explicit arithmetic
presentation of the fundamental group of the Cartwright--Steger
surface $X$, determined its holomorphic automorphism group
$\Aut_{\mathrm{hol}}(X)$, and described its Albanese fibration; the
latter provides a Lefschetz fibration $\alpha \colon X\to E$ where $E$
is the hexagonal elliptic curve $E=\C/(\Z+\omegaE\Z)$ with
$\omegaE=e^{2\pi i/3}$, see \cref{thm:three-nodes}. The existence of
an antiholomorphic involution on $X$ was first proved by Borisov and
Yeung \cite[Introduction and Remark~5.1]{BY}: their bicanonical
equations can be chosen over $\mathbb Q$, hence coordinatewise complex
conjugation preserves the surface and defines an antiholomorphic
involution $c_{\mathrm{BY}}$ on the Cartwright--Steger surface
$X$. The aim of the present paper is to determine the topology of
$X/\langle c_{\mathrm{BY}}\rangle $.

Let $\sigma$ generate $\Aut_{\mathrm{hol}}(X)\cong\Z/3\Z$.  An
antiholomorphic involution $c$ satisfying
$c\sigma c=\sigma^{-1}$ will be called a \emph{reflection}.  In the
following we construct an
arithmetic reflection $c_s$ on the complex ball with the property that
it descends to the Cartwright--Steger surface $X$, and together with
$\Aut_{\mathrm{hol}}(X)$, generates an $S_3$-action on $X$.

\begin{theorem}\label{thm:main}
Let $c$ be any reflection in the $S_3$-action on the
Cartwright--Steger surface $X$, and let $Y=X/\langle c\rangle$.  Then $Y$ is
a closed oriented smooth four-manifold, satisfying
\begin{enumerate}[label=\textup{(\alph*)}]
\item $\pi_1(Y)\cong\mathbb Z$.  Indeed, the Albanese map
  $\alpha \colon X\to E$ descends to
      \[
       \bar\alpha\colon Y\longrightarrow E/\langle z\mapsto\bar z\rangle ={\mathcal {M}},
      \]
where ${\mathcal {M}}$ is the M\"obius band, and $\bar\alpha_*$ is an
isomorphism $\pi _1(Y)\to \pi _1 ({\mathcal   {M}})\cong \Z$.
      
\item $Y$ is homeomorphic to $\CP^2\#(S^1\times S^3)$.
\end{enumerate}
\end{theorem}
The proof also provides the diffeomorphism type of the branch locus of the
branched cover $q_X\colon X\to Y$.

\begin{corollary}\label{cor:branched-cover}
The Cartwright--Steger surface $X$ is a double branched cover of a smooth
four-manifold $Y$ homeomorphic to $\CP^2\#(S^1\times S^3)$, branched along the
branch locus  diffeomorphic to the connected sum of three
real projective planes $\# _3 \RP ^2$.
\end{corollary}

\begin{remark}\label{rem:exotic}
  Indeed, $Y$ can be shown to be diffeomorphic to
  $\CP^2\#(S^1\times S^3)$, as it will be explained in our forthcoming work
  \cite{SSz}.
\end{remark}

In the computation of $\pi_1(Y)$ we use Armstrong's theorem
\cite{Armstrong} to get a cyclic upper bound for the fundamental
group, while the descended Albanese map gives a surjection onto
$\mathbb Z$.  One of the arguments is assisted by a \textsf{GAP}
program \cite{GAP4}.  Part~(b) of Theorem~\ref{thm:main} relies on the
Freedman--Quinn classification of closed oriented topological
four-manifolds with fundamental group $\Z$, for which the equivariant
intersection form over $\Lambda=\Z[t,t^{-1}]$ is a complete invariant
among the smooth ones.

\subsection*{Acknowledgements}
The first author was partially supported by the NKFIH Grant K146401
and by ERC Advanced Grant KnotSurf4d. The second author was partially
supported by the Simons Grant \emph{New structures in low dimensional
topology}.  The authors would like to thank Inan\c{c} Baykur for
helpful discussions.  The authors  acknowledge that
OpenAI’s ChatGPT and Anthropic’s Claude were used as research-support
tools for exploratory calculations and editorial revisions.
They also acknowledge the use of OpenAI's ChatGPT to help write
the GAP script given in the Appendix (and used in one of the arguments of
Section 3).
The outputs were not treated as mathematical authority, and the authors
assume full intellectual responsibility for the content.

\section{The arithmetic model of $X$ and its real structures}\label{sec:arithmetic}

We start by recalling the definition of the Cartwright--Steger surface and
listing some facts about it.

\subsection{The Cartwright--Steger lattice}

Let
$ \zeta=\zetaT=e^{\pi i/6}, r=\zeta+\zeta^{-1}=\sqrt3 $,
and consider the Hermitian form
\[
 F_0=
 \begin{pmatrix}
 r+1&-1&0\\
 -1&r-1&0\\
 0&0&-1
 \end{pmatrix}.
\]
By a \emph{line} we always mean a complex line through the origin in
$\C^3$, that is a point of $\CP^2$.  Such a line $L$ is called
\emph{negative}, \emph{null} or \emph{positive} according to the sign
of $F_0(v,v)=v^*F_0v$ for $0\ne v\in L$ (where $v^*={\overline
  {v}}^T$); since $F_0(\lambda v,\lambda v) =|\lambda|^2F_0(v,v)$,
that sign depends only on the line and not on the chosen
representative. Define
\[
{\mathrm {U}}(F_0)=\{ A\in {\mathrm {GL}}_3(\C )\mid A^*F_0A=F_0\}
\]
and ${\mathrm {PU}}(F_0)$ as its quotient by the center $Z({\mathrm {U}}(F_0))$:
\[
  {\mathrm {PU}}(F_0)={\mathrm {U}}(F_0)/\{\lambda {\mathrm {Id}}_3\mid |\lambda
  |=1\}.
\]

As the eigenvalues of $F_0$ are $\sqrt3\pm\sqrt2$ and $-1$, the form
$F_0$ has signature $(2,1)$.  Diagonalizing $F_0$ and passing to the
affine chart in which the last coordinate equals $1$ identifies the
negative lines with the standard unit complex ball in $\C^2$; they therefore
form a nonempty open subset of $\CP^2$ biholomorphic to the unit complex ball,
and this is the model of $\B$ we will use throughout.
Since the three classes are preserved
by any matrix preserving $F_0$, the group $\mathrm{PU}(F_0)$ acts on
$\B$, which is how the arithmetic lattice we describe below acts by isometries.
For example, the negative line $[0:0:1]$ is a point of $\B$; it will
be the base point when such a choice is needed in what follows.

Following \cite{CKY}, consider the matrices
\begin{align}
 u_0&=
 \begin{pmatrix}
 \zeta^3+\zeta^2-\zeta&1-\zeta&0\\
 \zeta^3+\zeta^2-1&\zeta-\zeta^3&0\\
 0&0&1
 \end{pmatrix},\label{eq:u0}\\[3pt]
 v_0&=
 \begin{pmatrix}
 \zeta^3&0&0\\
 \zeta^3+\zeta^2-\zeta-1&1&0\\
 0&0&1
 \end{pmatrix},\label{eq:v0}\\[3pt]
 b_0&=
 \begin{pmatrix}
 1&0&0\\
 -2\zeta^3-\zeta^2+2\zeta+2&
 \zeta^3+\zeta^2-\zeta-1&-\zeta^3-\zeta^2\\
 \zeta^2+\zeta&-\zeta^3-1&-\zeta^3+\zeta+1
 \end{pmatrix}.\label{eq:b0}
\end{align}
These matrices preserve $F_0$.  Let $u,v,b$ be their projective classes in
$\mathrm{PU}(F_0)$, and put $j=(uv)^2$; as a matrix
$j=\operatorname{diag}(\zeta,\zeta,1)$.  By
\cite[Theorem~1]{CKY}, which reproduces a result of Cartwright and Steger
\cite{CS2}, the ambient arithmetic lattice has the presentation
\[
 \GammaMax=
 \left\langle u,v,b\ \middle|\
 \begin{array}{c}
 u^3=v^4=b^3=1,\quad (uv)^2=(vu)^2,\quad vb=bv,\\
 (buv)^3=1,\quad (buvu)^2v=1
 \end{array}
 \right\rangle .
\]

Define
\begin{align}
 a_1&=vuv^{-1}j^4buvj^2,\\
 a_2&=v^2ubuv^{-1}uv^2j,\\
 a_3&=u^{-1}v^2uj^9bv^{-1}uv^{-1}j^8.
\end{align}
Multiplying the matrices of \eqref{eq:u0}--\eqref{eq:b0} and normalizing the
product to have determinant one (which determines it up to a scalar cube root
of unity) shows that the three elements $a_1,a_2,a_3$ can be represented by
\begin{align}
 a_1&=
 \begin{pmatrix}
 \zeta^3+\zeta^2    &-\zeta+1                &1\\
 2\zeta^3+2\zeta^2-1&-\zeta^3-\zeta^2+\zeta+1&-\zeta^3+2\zeta+1\\
 \zeta^2+\zeta      &-\zeta^3                &-\zeta^3-\zeta^2+\zeta+1
 \end{pmatrix},\label{eq:a1-matrix}\\[3pt]
 a_2&=
 \begin{pmatrix}
 -\zeta-1        &-\zeta^3+\zeta^2&\zeta^3\\
 \zeta^3-2\zeta-2&\zeta           &\zeta^2+\zeta\\
 -\zeta-1        &0               &\zeta^3+\zeta^2
 \end{pmatrix},\label{eq:a2-matrix}\\[3pt]
 a_3&=
 \begin{pmatrix}
 -\zeta^2-\zeta   &1      &-\zeta^3-1\\
 -\zeta^2-2\zeta-1&\zeta+1&-\zeta^2-\zeta\\
 \zeta^2+\zeta    &-1     &\zeta^2+\zeta
 \end{pmatrix}.\label{eq:a3-matrix}
\end{align}
Each of these matrices preserves $F_0$.
We use the following:

\begin{theorem}[Cartwright--Koziarz--Yeung, \cite{CKY}]\label{thm:CKY-input}
Let $\PiCS=\langle a_1,a_2,a_3\rangle<\GammaMax$.
\begin{enumerate}[label=\textup{(\alph*)}]
\item $\PiCS$ is torsion free of index $864$, and
  $X=\PiCS\backslash\B$ is the Cartwright--Steger surface.  It is a smooth
  compact complex surface, and has
      Euler characteristic $\chi(X)=3$ and signature $\sgn(X)=1$.
\item There is an isomorphism $H_1(X;\Z)\cong\Z^2$ for which
      \begin{equation}\label{eq:H1-vectors}
      [a_1]=(1,3),\qquad [a_2]=(-2,1),\qquad [a_3]=(-1,-1).
      \end{equation}
      Here each $a_i$ denotes the corresponding element in the
      fundamental group, and $[a_i]$ is its image in the first
      homology.
\item The normalizer of $\PiCS$ in $\GammaMax$ is generated by $\PiCS$
  and $j^4$.  The automorphism $\sigma$ induced by $j^4$ generates
  $\Aut_{\mathrm{hol}}(X)\cong\Z/3\Z$, and in the 
  convention of \eqref{eq:H1-vectors} its action on $H_1(X; \Z )$ is given by
      \begin{equation}\label{eq:sigma-H1}
      \sigma_*=A:=\begin{pmatrix}0&1\\-1&-1\end{pmatrix}.
      \end{equation}
    \item The Albanese variety $\Alb(X)$ is the hexagonal elliptic curve $E$,
      and the map $\alpha _*\colon \pi_1(X)\to\pi_1(\Alb(X))$ induced by the Albanese map
      $\alpha\colon X\to E$ is onto.
\end{enumerate}
\end{theorem}
\begin{proof}
Parts (a) and (b) are \cite[Theorem~2]{CKY}.  The normalizer and its
action on homology are recorded in \cite[\S1.4]{CKY}, immediately
after that theorem, in the row-vector convention $f(j^4\pi
j^{-4})=f(\pi)\left(\begin{smallmatrix}0&-1\\1&-1\end{smallmatrix}\right)$;
  \eqref{eq:sigma-H1} is its transpose.  That
  $\Aut_{\mathrm{hol}}(X)\cong\Z/3\Z$ is part~(f) of the Main Theorem of
  \cite{CKY}.  For (d), the Albanese torus is identified in
  \cite[Lemma~11]{CKY}, and in \cite[\S5.4]{CKY} the period
  homomorphism is normalized to $(m,n)\mapsto m-n\omegaE$ (with
  $\omegaE=e^{2\pi i/3}$), which is an
  isomorphism onto $\Z+\omegaE\Z$.
\end{proof}

\subsection{An antiholomorphic involution}
Borisov and Yeung's rational model already supplies the involution
$c_{\mathrm{BY}}$.  We now construct an arithmetic representative of
an antiholomorphic involution and verify its properties directly.

The involution $s$ will be defined on $\B$ so that it
fixes the negative line $[0:0:1]\in \B\subset \CP ^2$ (which will serve as base point later on) and 
descends to an
antiholomorphic involution on $X=\PiCS\backslash\B$.
The first idea might be simple
complex conjugation, but that does not descend to $X$.
Consider the element
$ k=v^{-1}uv^{-1}u^{-1}v\in\GammaMax$ ;
in matrix terms, we get
\[
 v_0^{-1}u_0v_0^{-1}u_0^{-1}v_0=-k,
 \qquad\text{where}\qquad
 k=
 \begin{pmatrix}
 -1&-1+\zeta+\zeta^2-\zeta^3&0\\
 0&\zeta^3&0\\
 0&0&-1
 \end{pmatrix},
\]
and from now on $k$ denotes the displayed matrix.
(Recall that $\zeta$ is a primitive 12$^{th}$ root of unity.)
For a vector $w\in \C ^3$ representing a negative line $L\in \B$ we define
the map $s\colon \B\to \B$ by the formula
\[
 s(w)=k {\overline {w}}, 
\]
where $k$ acts by matrix
multiplication.

\begin{proposition}\label{prop:antiunitary}
The transformation $s$ is an antiholomorphic isometry of $\B$, satisfies
$s^2=1$, and fixes the negative line
$[0:0:1]\in\B$.
\end{proposition}

\begin{proof}
As $k$ is $\mathbb C$-linear and conjugation is anti-$\mathbb
C$-linear, we get that $s$ is anti-$\mathbb C$-linear.  The identities
\[
 u_0^*F_0u_0=v_0^*F_0v_0=b_0^*F_0b_0=F_0
\]
imply $k^*F_0k=F_0$, and ${\overline {F_0}}=F_0$ since entries of
$F_0$ are all real.  Therefore $s$ preserves the Hermitian form in the
antiunitary sense and sends negative lines to negative lines.  Its
projectivization on $\B$ is consequently antiholomorphic.

It remains to check that it is an involution.  If
$q=-1+\zeta+\zeta^2-\zeta^3$, then
$\bar q=\zeta-\zeta^2=q\zeta^{-3}$, and matrix multiplication gives
\[
 k{\overline {k}}=
 \begin{pmatrix}
 1&-\bar q+q\zeta^{-3}&0\\0&1&0\\0&0&1
 \end{pmatrix}=I.
\]
Thus $s^2=1$ as claimed.  Finally,
$k(0,0,1)^T=-(0,0,1)^T$ and conjugation fixes $(0,0,1)^T$, whose entries are all
real.  The line spanned by this vector is therefore
fixed by $s$.
\end{proof}
Next we will show that $s$ descends to an involution on $X$. For this
we need to show that for each group element $g\in \GammaMax$ there is
$\phiA (g)$ so that if $g$ takes $w\in \B$ to  $w'\in \B$ then
$\phiA (g)$ takes $s(w)$ to $s(w')$. The formula 
\begin{equation}\label{eq:phi-def}
 \phiA(g)=sgs^{-1}
\end{equation}
will exactly do that --- indeed, as $s$ is an involution, $s^{-1}=s$.

\begin{proposition}\label{prop:phi-words}
The following identities hold:
\begin{align}
 \phiA(a_1)&=a_1^{-1}a_2^{-1}a_1^{-1}a_3^{-3}a_2^3,
 \label{eq:phi-a1}\\
 \phiA(a_2)&=a_1^{-1}a_2^{-1}a_1a_2^2
 a_1^{-1}a_2^{-1}a_1a_3^{-1}a_1^{-1}a_2a_1,
 \label{eq:phi-a2}\\
 \phiA(a_3)&=a_3^{-1}.
 \label{eq:phi-a3}
\end{align}
Consequently $\phiA(\PiCS)=\PiCS$, and $s$ descends to an
antiholomorphic involution $c_s$ on $X$.
\end{proposition}

\begin{proof}
Using \eqref{eq:a1-matrix}--\eqref{eq:a3-matrix} and
\eqref{eq:phi-def}, both sides of each displayed identity are
evaluated in $\operatorname{Mat}_3(\Z[\zeta ])$ as matrices. The two
sides agree entry by entry after multiplication by an appropriate
$\zeta$-power (which is $\zeta ^8, \zeta ^8, \zeta ^0$, respectively);
hence the identities hold in $\GammaMax$.  The identities give
$\phiA(\PiCS)\subseteq\PiCS$; since $s^2=1$ and so $\phiA^2=\id$,
equality follows.
\end{proof}

\begin{proposition}\label{prop:S3}
The projective transformation $j^4$ has order three and
\begin{equation}\label{eq:S3}
 sj^4s^{-1}=j^{-4}.
\end{equation}
Thus $c_s$ and $\sigma$ generate an $S_3$-action on $X$.  The three elements
$c_s\sigma^m$, $m=0,1,2$, are antiholomorphic involutions and are conjugate by
powers of $\sigma$.
\end{proposition}

\begin{proof}
Recall that $j=(uv)^2$ in matrix form is represented by 
$j=\operatorname{diag}(\zeta,\zeta,1)$.  Therefore matrix multiplication gives
$(j^4)^3=1$, $j^4\ne1$, and $sj^4s^{-1}=j^{-4}$.  Hence
$c_s^2=\sigma^3=1$ and $c_s\sigma c_s=\sigma^{-1}$.  It follows that
$(c_s\sigma^m)^2=1$ and
\[
 \sigma^n c_s\sigma^{-n}=c_s\sigma^{-2n}.
\]
As $n$ runs through $\mathbb Z/3\Z$, each of the three reflections
occurs; therefore, they are conjugate to one another.
\end{proof}

\begin{corollary}\label{cor:BY-comparison}
Every antiholomorphic automorphism of $X$ is of the form $c_s\sigma^m$ for a
unique $m\in\Z/3\Z$.  In particular,
\begin{equation}\label{eq:BY-comparison}
 c_{\mathrm{BY}}=c_s\sigma^m
\end{equation}
for some $m$.  Thus the Borisov--Yeung involution and the arithmetic
involution used here are the same modulo the holomorphic $\Z/3\Z$-action,
and their orbit spaces are diffeomorphic.
\end{corollary}

\begin{proof}
If $d$ is antiholomorphic, then $c_s^{-1}d$ is holomorphic.  By
\cref{thm:CKY-input}, it equals $\sigma^m$ for some $m$, so
$d=c_s\sigma^m$.  Applying this to $d=c_{\mathrm{BY}}$ gives
\eqref{eq:BY-comparison}.  Finally, conjugation by an appropriate
power of $\sigma$ carries one reflection to any other and therefore
induces a diffeomorphism of the corresponding orbit spaces.
\end{proof}

From now on we work with the arithmetic reflection $c=c_s$.

\begin{proposition}\label{prop:H1-action}
  In the coordinates \eqref{eq:H1-vectors}, the reflection $c$ acts on
  $H_1(X; \Z)$ by
\begin{equation}\label{eq:M}
 c_*=M:=\begin{pmatrix}0&-1\\-1&0\end{pmatrix}.
\end{equation}
\end{proposition}

\begin{proof}
The exponent sums of \eqref{eq:phi-a1} are $(-2,2,-3)$, so its homology class
is
\[
 -2(1,3)+2(-2,1)-3(-1,-1)=(-3,-1).
\]
For \eqref{eq:phi-a2} the exponent sums are $(0,1,-1)$, giving $(-1,2)$,
and \eqref{eq:phi-a3} gives $(1,1)$.  These images determine the matrix 
$M$.
\end{proof}

Let $x_0\in X$ be the image of the negative line $[0:0:1]$, which lies in
$\B$. By \cref{prop:antiunitary},
it is fixed by $s$, 
and $j^4=\operatorname{diag}(\omegaE,\omegaE,1)$ also fixes the same line;
hence $x_0$ is fixed by $c$ and by $\sigma$.

\begin{proposition}\label{prop:real-Albanese}
There is an identification
\[
 E\cong\C/(\Z+\omegaE\Z)
\]
for which
\[
 \sigma_E(z)=\omegaE z,
 \qquad
 c_E(z)=\bar z,
\]
and $\alpha$ is equivariant for both actions.
\end{proposition}

\begin{proof}
Albanese functoriality gives a holomorphic automorphism $\sigma_E$ of
$E$ with $\alpha\sigma=\sigma_E\alpha$.  For $c$ one applies the same
functoriality, where $c$ is viewed as a holomorphic map $X\to\overline
X$, which yields an antiholomorphic automorphism $c_E$ of $E$ with
$\alpha c=c_E\alpha$.  Since $x_0$ is fixed by $\sigma$ and by $c$ and
$\alpha(x_0)=0$, both $\sigma_E$ and $c_E$ fix $0$ and are therefore
group automorphisms, with no translation term.  By
\cref{thm:CKY-input}\textup{(d)} the map $\alpha_*\colon H_1(X;\Z)\to
H_1(E;\Z)$ is a surjection between free abelian groups of rank two,
hence an isomorphism, and it intertwines the two actions.  So
$\sigma_E$ and $c_E$ are given on $H_1(E;\Z)$ by the matrices $A$ of
\eqref{eq:sigma-H1} and $M$ of \eqref{eq:M}; in the integral basis
determined by
\[
 P=\begin{pmatrix}-1&1\\1&0\end{pmatrix}
\]
these become the matrices $P^{-1}AP$ and $P^{-1}MP$.
A direct calculation gives that $ P^{-1}AP$ is multiplication by
$\omegaE$ in the lattice basis $(1,\omegaE)$, while $ P^{-1}MP$ is
complex conjugation, since $\bar\omegaE=-1-\omegaE$.
\end{proof}

\section{The computation of the fundamental group}
\label{sec:coequalizer}

Throughout this section $Y=X/\langle c\rangle\cong\Gamma_{\mathbb
  R}\backslash \B$, where
\begin{equation}\label{eq:GammaR}
 \Gamma_{\mathbb R}=\PiCS\rtimes_{\phiA}\langle s\mid s^2=1\rangle
\end{equation}
is the group acting on the ball; the semidirect product makes sense
as by \cref{prop:phi-words}
$s$ normalizes $\PiCS$. 

\subsection{The Albanese lower bound}
\label{sec:lower-bound}

We first show that there is a surjection $\pi_1(Y)\to\Z$.
Write a point of $E=\C/(\Z+\omegaE\Z)$ as $x+y\omegaE$, with
$(x,y)\in\mathbb R^2/\Z^2$.  Complex conjugation acts by
\[
 c_E(x,y)=(x-y,-y).
\]
Its fixed locus is the single circle $y=0$, so the quotient
$\mathcal M=E/\langle c_E\rangle$ is the M\"obius band, with the factorization
map $q_E\colon E\to {\mathcal {M}}$.

Let $e_1,e_2$ be the lattice loops represented by $1$ and $\omegaE$, and let
$t_{\mathcal M}$ be the core of $\mathcal M$.  The loop $e_1$ becomes the
boundary of the M\"obius band, so, after orienting the core,
\[
 (q_E)_*(e_1)=2t_{\mathcal M}.
\]
Equivariance and $(c_E)_*(e_2)=-e_1-e_2$ imply
$(q_E)_*(e_2)=-t_{\mathcal M}$.  Consequently
\begin{equation}\label{eq:qE-surj}
 (q_E)_*\colon \pi_1(E)=\Z^2\twoheadrightarrow\pi_1(\mathcal M)=\Z
\end{equation}
is surjective.

By \cref{prop:real-Albanese}, the Albanese map descends to a commutative
diagram
\begin{equation}\label{eq:diagram}
\begin{tikzcd}[column sep=large,row sep=large]
 X \arrow[r,"\alpha"] \arrow[d,"q_X"'] &
 E \arrow[d,"q_E"]\\
 Y \arrow[r,"\bar\alpha"'] & \mathcal M.
\end{tikzcd}
\end{equation}
Since $\alpha_*\colon \pi_1(X)\twoheadrightarrow\pi_1(E)$ is surjective by
\cref{thm:CKY-input}, equations \eqref{eq:qE-surj} and \eqref{eq:diagram}
give the following:

\begin{proposition}\label{prop:lower}
The descended Albanese map induces a surjection
\[
 \bar\alpha_*\colon \pi_1(Y)\twoheadrightarrow\pi_1(\mathcal M)=\Z.
\]
\end{proposition}

\subsection{The coequalizer of the reflection}
\label{sec:coequalizer-sub}

Our next aim is to show that $\pi_1(Y)$ is cyclic.
We first recall the group-theoretic construction used below.

\begin{definition}\label{def:coequalizer}
  {\it Let $f,g\colon A\rightrightarrows B$ be two homomorphisms.
    Their {\bf{\emph{coequalizer}}} in the category of groups is
\[
 \operatorname{coeq}(f,g)
 =B/\normal{f(a)g(a)^{-1}\mid a\in A}.
\]
It comes with a quotient homomorphism $q\colon B\to\operatorname{coeq}(f,g)$ for
which $qf=qg$.  It is universal with this property: whenever
$h\colon B\to C$ satisfies $hf=hg$, there is a unique
homomorphism $\bar h\colon \operatorname{coeq}(f,g)\to C$ with $h=\bar hq$.}
\end{definition}

In the present situation we compare the automorphisms $\phiA$ and $\id$ of
$\PiCS$. Their coequalizer is
\[
 Q=\operatorname{coeq}(\phiA,\id)
  =\PiCS/\normal{\phiA(g)g^{-1}:g\in\PiCS},
\]
the largest quotient of $\PiCS$ on which the reflection acts trivially.

\begin{lemma}\label{lem:kill-s}
$\Gamma_{\mathbb R}/\normal{s}\cong Q$.
\end{lemma}

\begin{proof}
Killing $s$ in \eqref{eq:GammaR} imposes exactly the relations
$\phiA(g)=sgs^{-1}=g$ for $g\in\PiCS$, and conversely $\normal{s}$ is
generated by $s$ together with these relations.\end{proof}

As $a_1,a_2,a_3$ generate $\PiCS$, in presenting $Q$ it is enough to impose the three
relations
\[
 \phiA(a_i)a_i^{-1}=1,\qquad i=1,2,3.
\]
After these relations are imposed, $\phiA$ and $\id$ agree on a generating set
and therefore on all of the quotient.

\begin{proposition}\label{prop:Q-Z}
The reflection coequalizer is infinite cyclic.  More precisely,
\begin{equation}\label{eq:marked-Q-presentation}
 Q\cong
 \left\langle a_1,a_2,a_3\ \middle|\
 a_3,\ [a_1,a_2],\ a_1^3a_2^{-2}
 \right\rangle ,
\end{equation}
showing that $Q\cong\Z$, generated by $t=a_1^{-1}a_2$.
\end{proposition}

\begin{proof}
The fact that $Q$ has the presentation of
Equation~\eqref{eq:marked-Q-presentation} has been checked directly in
\textsf{GAP} \cite{GAP4}.  A short program is reproduced in
\cref{app:gap}.

Now the group presented by \eqref{eq:marked-Q-presentation} is abelian, and
a simple argument shows that it is isomorphic to $\Z$.
\end{proof}

\subsection{The upper bound}
\label{sec:upper-bound}

We use Armstrong's orbit-space theorem to show that $\pi _1(Y)$ is
a quotient of $\Z$.

\begin{theorem}[Armstrong \cite{Armstrong}]\label{thm:Armstrong}
Let the discrete group $G$ act discontinuously on a simply connected, locally
compact, path-connected metric space $Z$.  If $N$ is the normal subgroup
generated by the elements of $G$ having a fixed point in $Z$, then
\[
 \pi_1(Z/G)\cong G/N.
\]
\end{theorem}

The hypotheses apply to $\Gamma_{\mathbb R}$ acting on $\B$.  Recall that by
\cref{prop:antiunitary}, $s$ has a fixed point in $\B$.

\begin{proposition}\label{prop:upper}
There is a surjection
\[
 Q\twoheadrightarrow\pi_1(Y).
\]
In particular, $\pi_1(Y)$ is a quotient of $\mathbb Z$.
\end{proposition}

\begin{proof}
Armstrong's theorem gives $\pi_1(Y)=\Gamma_{\mathbb R}/N$, where
$s\in N$.  Hence the quotient map factors through
$\Gamma_{\mathbb R}/\normal{s}\cong Q$ by \cref{lem:kill-s},
and by \cref{prop:Q-Z} we have 
$Q\cong\mathbb Z$.
\end{proof}

Now we can complete the computation of the fundamental group of $Y$:
\begin{proof}[Proof of \cref{thm:main}\textup{(a)}]
For the arithmetic reflection $c=c_s$, \cref{prop:upper} shows that
$\pi_1(Y)$ is a quotient of $Q\cong\mathbb Z$.  Every proper quotient of
$\mathbb Z$ is finite cyclic, whereas \cref{prop:lower} gives a surjection
$\pi_1(Y)\twoheadrightarrow\mathbb Z$, therefore
$ \pi_1(Y)\cong\Z$.

The surjective homomorphism $\bar\alpha_*\colon \Z\to\Z$ is consequently an
isomorphism.  The remaining two reflections are conjugate to $c_s$ by powers
of $\sigma$, and these conjugacies induce diffeomorphisms of the corresponding
orbit spaces and compatible automorphisms of the Albanese torus.  
\end{proof}

From now on we fix the isomorphism $\pi_1(Y)\cong\Z$ given by
$\bar\alpha_*$, and write $t$ for the corresponding generator.

\section{Completion of the proof of \cref{thm:main}}
\label{sec:quotient-manifold}

Write $q_X\colon X\to Y$ for the quotient map and
\[
 F=X^c=\Fix(c)
\]
for the real fixed surface.  Since $c$ is antiholomorphic, its differential
at a fixed point is an anti-$\C$-linear involution of $T_pX\cong\C^2$, whose
fixed subspace is a real form.  Hence $F$ is a smooth totally real surface
of real dimension two at every point; it is nonempty, because $s$ fixes the
negative line $[0:0:1]$ by \cref{prop:antiunitary}.

The orbit space $Y$ is a closed connected oriented topological four-manifold,
and it carries a smooth structure for which $q_X$ is a two-fold branched
covering with branch surface $q_X(F)$.  We orient $Y$ so that $q_X$ is
orientation preserving, which is possible because an antiholomorphic map
preserves the underlying real orientation of a complex surface.

\begin{remark}
Averaging a Riemannian metric over $\langle c\rangle$ gives, at each point of
$F$, normal coordinates in which $c(x_1,x_2,y_1,y_2)=(x_1,x_2,-y_1,-y_2)$,
with $F=\{y_1=y_2=0\}$, so the local orbit space is
$\R^2\times(\R^2/\pm1)$ and the squaring map $w\mapsto w^2$ identifies
$\R^2/\pm1$ with $\R^2$.  Taking these as charts near $q(F)$ and $q$ itself
as a chart away from it gives a smooth structure, because squaring turns the
orthogonal transition maps on the normal directions into rotations and
reflections again.
\end{remark}

The Euler characteristic of the fixed surface
$F$ will be determined by the singular fibers of the
Albanese fibration together with the $S_3$-symmetry.  We first record what
is needed from the literature.

\begin{theorem}[Cartwright--Koziarz--Yeung, Koziarz--Yeung, Rito, \cite{CKY, KY, Rito}]
\label{thm:three-nodes}
The Albanese map $\alpha \colon X\to E$ has exactly three critical points.  Each is
an ordinary node of its fiber. The three lie in three distinct fibers, and
they form a single free orbit of $\langle\sigma\rangle\cong\Z/3\Z$.
\end{theorem}

By \cref{prop:real-Albanese} the fixed locus of $c_E$ on
$E=\C/(\Z+\omegaE\Z)$ is the circle $E^{c_E}=\R/\Z$, so
the restriction  $\alpha _{\R}:=\alpha |_F$ is a  smooth circle-valued map
\[
 \alpha_{\R}\colon F\longrightarrow E^{c_E}\cong S^1 .
\]

\begin{lemma}\label{lem:real-critical}
A point $p\in F$ is critical for $\alpha_\R$ if and only if it is critical
for $\alpha$.
\end{lemma}

\begin{proof}
At $p\in F$ the tangent space decomposes as $T_pX=T_pF\oplus J(T_pF)$, and
$d\alpha_p$ is complex linear.  If $d\alpha_p$ vanishes on $T_pF$ then for
every $v\in T_pF$ we get $d\alpha_p(Jv)=J\,d\alpha_p(v)=0$, so $d\alpha_p=0$
on all of $T_pX$.  The converse is immediate.
\end{proof}

\begin{lemma}\label{lem:one-real-node}
Exactly one of the three nodes lies in $F$, so $\alpha_\R$ has exactly one
critical point.
\end{lemma}

\begin{proof}
Since $\alpha c=c_E\alpha$, the involution $c$ permutes the three
nodes; let ${\mathcal {N}}=\{n,\sigma n,\sigma^2 n\}$ be that orbit,
which is free by \cref{thm:three-nodes}.  An involution of a
three-element set is either the identity or a transposition.  Suppose
$c$ fixed every element of ${\mathcal {N}}$.  Using $c\sigma
c=\sigma^{-1}$ from \eqref{eq:S3},
\[
 \sigma n=c(\sigma n)=\sigma^{-1}c(n)=\sigma^{-1}n,
\]
so $\sigma^2 n=n$, contradicting freeness.  Hence $c$ is a
transposition on ${\mathcal {N}}$ and fixes exactly one node, which
therefore lies in $F$.  Now apply \cref{lem:real-critical}.
\end{proof}

\begin{lemma}\label{lem:index-one}
The unique critical point of $\alpha_\R$ is nondegenerate of index one.
\end{lemma}

\begin{proof}
Let $p$ be the critical point and $y_0=\alpha(p)$.  Since $\alpha$ is
holomorphic and $d\alpha_p=0$, its second derivative at $p$ is a well-defined
symmetric $\C$-bilinear form
\[
 H\colon T_pX\times T_pX\longrightarrow T_{y_0}E,
\]
the quadratic term of $\alpha$ in any local holomorphic coordinates:
$\alpha=y_0+\tfrac12H(z,z)+O(|z|^3)$.  That the fiber of $\alpha$ through $p$
has an ordinary node there says precisely that $H$ is nondegenerate.

This nondegeneracy is inherited by the
restriction to $F$:  Since $\alpha(F)\subseteq E^{c_E}$ and $\alpha_\R$ is the
restriction of $\alpha$, and since $d(\alpha_\R)_p=0$ by
\cref{lem:real-critical}, the Hessian of $\alpha_\R$ at $p$ is the
restriction of $H$ to $T_pF\times T_pF$, with values in the real line
$T_{y_0}E^{c_E}$.  Now $F$ is totally real, so $T_pX=T_pF\oplus J\,T_pF$ and
a real basis $e_1,e_2$ of $T_pF$ is at the same time a $\C$-basis of $T_pX$;
likewise a generator of $T_{y_0}E^{c_E}$ is a $\C$-basis of $T_{y_0}E$.  With
respect to these choices the real symmetric matrix
$\bigl(H(e_i,e_j)\bigr)$ of the Hessian of $\alpha_\R$ is literally the
matrix of $H$, so the two determinants are equal.  It is nonzero,
and $\alpha_\R$ is therefore Morse at $p$, of index $0$, $1$ or $2$.

Let $x\in S^1$ be the critical value.  Over $S^1\setminus\{x\}$ the map
$\alpha_\R$ is a proper submersion, hence a locally trivial fibration by
Ehresmann's theorem \cite{Ehresmann}; since $S^1\setminus\{x\}$ is connected, all of its
fibers are diffeomorphic, and in particular the regular level sets on the two
sides of $x$ have the same number of circle components.  A critical point of
index $0$ creates a circle component and one of index $2$ destroys one, so
either would change that number.  Only index one remains.
\end{proof}

\begin{proposition}\label{prop:chi-F}
The real fixed surface satisfies
\begin{equation}\label{eq:chi-F}
 \chi(F)=-1 .
\end{equation}
\end{proposition}

\begin{proof}
Pull back the nowhere-zero closed one-form $d\theta$ on $S^1$ and take a
gradient-like vector field for the circle-valued Morse function
$\alpha_\R$.  Its zeros are the critical points of $\alpha_\R$, with index
$(-1)^{\operatorname{ind}}$, so the Poincar\'e--Hopf theorem gives
$\chi(F)=\sum_p(-1)^{\operatorname{ind}(p)}$.  By
\cref{lem:one-real-node,lem:index-one} there is one critical point and its
index is one.
\end{proof}

\begin{corollary}\label{cor:chi-Y}
  For the Euler characteristic of $Y$ we have $\chi(Y)=1$.  Therefore
  $H_i(Y;\Z)\cong\Z$ for $0\le i\le4$, and in particular
  $H_2(Y;\Z)\cong\Z$. 
\end{corollary}

\begin{proof}
The involution acts freely on $X\setminus F$, so
\[
 \chi(Y)=\tfrac12\bigl(\chi(X)-\chi(F)\bigr)+\chi(F).
\]
As $\chi(X)=3$ and $\chi(F)=-1$ by \eqref{eq:chi-F}, we get $\chi(Y)=1$.
Since $\pi_1(Y)\cong\Z$ we have $b_0=b_1=b_3=b_4=1$, hence $b_2(Y)=1$. Since
$H_1(Y;\Z)\cong\Z$ is torsion free, by Poincar\'e duality $H_2(Y;\Z)$ is
torsion free as well.  The claim on the homology groups follows.
\end{proof}

The above calculation then determines the signature of $Y$ as well.
Indeed, for the totally real submanifold $F\subset X$ we have that the
complex structure identifies the normal bundle of $F\subset X$ with
$TF$, but the ambient orientation convention reverses the induced
orientation, so the normal Euler number is
$e(\nu_{F/X})=-\chi(F)=1$.
By \cite{JO} this implies that the
equivariant signature $\sgn(c, X)$ is equal to 1, see also \cite{Hirzebruch}.

\begin{proposition}\label{prop:intformZ}
  The signature $\sgn (Y)$ is equal to 1, hence the intersection form of $Y$ is $Q_Y=\langle 1\rangle $.
  \end{proposition}
\begin{proof}
  From the $G$-signature theorem for $G=\Z/2\Z$ (see
  \cite{Hirzebruch}) we have that  $\sgn (Y)=\frac{1}{2}(\sgn (X)+
  \sgn (c, X))=1$. Since $H_2(Y; \Z )\cong \Z$, it implies the statement
  on the intersection form.
  \end{proof}

\subsection{The equivariant intersection form}
\label{sec:equivariant-form}

By the combination of theorems of Freedman \cite{Fr} and Donaldson
\cite{Do}, a closed \emph{simply connected} smooth four-manifold is determined up
to homeomorphism by its Euler characteristic and its signature, together
with the parity of its intersection form.  For manifolds with nontrivial
fundamental group the classification is more delicate, and one has to
understand the \emph{equivariant} intersection form.

Let
\[
 \Lambda=\Z[t,t^{-1}],\qquad \overline t=t^{-1},
\]
where $t$ generates $\pi_1(Y)$ as fixed at the end of
\cref{sec:coequalizer}.  Write $p\colon \widetilde Y\to Y$ for
the universal cover, so that as $\Lambda$-modules we have
$H_2(Y;\Lambda)=H_2(\widetilde Y;\Z)$ with $t$ acting as the deck
translation. Let
\begin{equation}\label{eq:lambda-def}
 \lambda_Y(x,y)=\sum_{n\in\Z}\bigl(x\cdot t^ny\bigr)\,t^n
\end{equation}
be the equivariant intersection pairing, where $x\cdot z$ denotes the
ordinary intersection number in $\widetilde Y$.  The sum is finite because
$x$ and $y$ are represented by compact cycles and the deck action is properly
discontinuous.  As the intersection number is symmetric and invariant under
deck translations, $\lambda_Y$ is $\Lambda$-linear in $x$, conjugate linear
in $y$, and satisfies $\lambda_Y(y,x)=\overline{\lambda_Y(x,y)}$.  The
classification result of four-manifold topology we use is the following.

\begin{theorem}[Freedman--Quinn]\label{thm:FQ}
Let $N$ be a closed oriented smooth four-manifold with $\pi_1(N)\cong\Z$.
Then $H_2(N;\Lambda)$ is a finitely generated free $\Lambda$-module and the
equivariant intersection pairing $\lambda_N$ on it is nonsingular Hermitian.
If moreover $\lambda_N\cong\langle1\rangle$, then $N$ is homeomorphic to
$\CP^2\#(S^1\times S^3)$.
\end{theorem}

This is the classification of \cite[Chapter~10]{FQ}, valid because $\Z$ is a
good group in the sense of Freedman, see also 
\cite[Theorem~3.1]{HambletonSurvey}.

For $\CP^2\#(S^1\times S^3)$ the form is easy to read off.  Its
universal cover is obtained from $\C^2\setminus\{0\}$
by taking the connected sum with one copy of 
$\CP^2$ for each deck translate;
the projective lines in these copies form a $\Z$-basis of $H_2$,
distinct translates are disjoint and each has self-intersection $+1$.  Hence
as a $\Lambda$-module we have
\begin{equation}\label{eq:lambda-model}
 \lambda_{\CP^2\#(S^1\times S^3)}=\langle 1\rangle .
\end{equation}

\begin{proposition}\label{prop:rank-one}
As $\Lambda$-modules we have $H_2(Y;\Lambda)\cong\Lambda$.
\end{proposition}

\begin{proof}
Let $K=\Q(t)$ be the quotient field of $\Lambda$; it is flat over $\Lambda$.
Since $t-1$ is invertible in $K$ we get $H_0(Y;K)=K/(t-1)K=0$ (as homology with
local coefficients). Furthermore,
$H_1(Y;\Lambda)=H_1(\widetilde Y;\Z)=0$ because $\widetilde Y$ is simply
connected, so $H_1(Y;K)=0$.  Poincar\'e duality with local coefficients
\cite[\S3.H]{Hatcher} then gives $H_3(Y;K)\cong\overline{H^1(Y;K)}=0$ and
$H_4(Y;K)\cong\overline{H^0(Y;K)}=0$, as over the  field $K$ we get
$H^i(Y;K)\cong\operatorname{Hom}_K(H_i(Y;K),K)$.  The cellular chain complex
of $\widetilde Y$ is finite and free over $\Lambda$, so its Euler characteristic
is $\chi(Y)$ for any coefficients, and therefore
\[
 \dim_K H_2(Y;K)=\chi(Y)=1 .
\]
By \cref{thm:FQ} the module $H_2(Y;\Lambda)$ is free, and its rank equals
$\dim_K H_2(Y;K)=1$, implying the result.
\end{proof}

\begin{lemma}\label{lem:rank-one-forms}
A nonsingular Hermitian form $\lambda$ on a free $\Lambda$-module of rank one
has matrix $\langle 1\rangle $ or $\langle -1\rangle $.
\end{lemma}

\begin{proof}
Let $e$ be a basis of the module and set $u=\lambda(e,e)$.  In this basis
$\lambda$ is the $1\times1$ matrix $(u)$, and nonsingularity says exactly
that this matrix is invertible over $\Lambda$, that is, $u$ is a unit.
As the units of $\Lambda=\Z[t,t^{-1}]$ are the elements $\pm t^n$, we
get $u=\pm t^n$ for some $n$.  The involution fixes $\Z$ and sends
$t\mapsto t^{-1}$, so $\overline u=\pm t^{-n}$ with the same sign, and
Hermitian symmetry gives $u=\overline u$, that is $t^n=t^{-n}$.  Hence
$n=0$ and $u=\pm1$.
\end{proof}
Define the \emph{augmentation} map $\varepsilon\colon \Lambda\to\Z$ by
sending $t\mapsto1$.

\begin{lemma}\label{lem:augmentation}
  The natural map
  $H_2(Y;\Lambda)\otimes_{\Lambda,\varepsilon}\Z\to H_2(Y;\Z)$ is an
  isomorphism, and it carries $\lambda_Y$ to the ordinary intersection
  form $Q_Y$.
\end{lemma}

\begin{proof}
  Give $Y$ a finite CW structure and let $C_*(\widetilde Y)$ be the
  cellular chain complex of the universal cover ${\widetilde {Y}}$, a
  finite complex of free $\Lambda$-modules with $H_*(C_*(\widetilde
  Y))=H_*(Y;\Lambda)$ and $C_*(\widetilde
  Y)\otimes_{\Lambda,\varepsilon}\Z=C_*(Y)$.  Multiplication by $t-1$
  is injective on $\Lambda$ with cokernel $\Z$. As $C_*(\widetilde Y)$
  is free, with the covering map $p\colon {\widetilde {Y}}\to Y$
  inducing $p_{\#}$ on chains, we have that
\[
 0\longrightarrow C_*(\widetilde Y)\xrightarrow{\ t-1\ }C_* (\widetilde Y)
 \xrightarrow{\ p_\#\ }C_*(Y)\longrightarrow0
\]
is a short exact sequence of chain complexes.  Its homology long exact
sequence is
\[
 \cdots\to H_n(Y;\Lambda)\xrightarrow{\ t-1\ }H_n(Y;\Lambda)
 \xrightarrow{\ p_*\ }H_n(Y;\Z)\to H_{n-1}(Y;\Lambda)
 \xrightarrow{\ t-1\ }H_{n-1}(Y;\Lambda)\to\cdots .
\]
Since $H_1(Y;\Lambda)=H_1(\widetilde Y;\Z)=0$, the portion at $n=2$ reads
\[
 H_2(Y;\Lambda)\xrightarrow{\ t-1\ }H_2(Y;\Lambda)
 \xrightarrow{\ p_*\ }H_2(Y;\Z)\longrightarrow0 ,
\]
so $p_*$ is onto with kernel $(t-1)H_2(Y;\Lambda)$ and therefore induces an
isomorphism
\[
 H_2(Y;\Lambda)\otimes_{\Lambda,\varepsilon}\Z
 =H_2(Y;\Lambda)/(t-1)H_2(Y;\Lambda)
 \xrightarrow{\ \cong\ }H_2(Y;\Z),
\]
which is the map of the statement.

For the pairings, let $x,y\in H_2(Y;\Lambda)=H_2(\widetilde Y;\Z)$ be
represented by compact cycles with $x$ transverse to $t^ny$ for every $n$.
The covering $p$ is a local orientation-preserving homeomorphism, and two
points of $\widetilde Y$ have the same image if and only if they differ by a
deck translation.  Hence the transverse intersections of $p_*x$ and $p_*y$
in $Y$ correspond bijectively, and with matching signs, to the pairs
consisting of an integer $n$ and a point of $x\cap t^ny$.  Counting with
signs,
\[
 p_*x\cdot p_*y=\sum_{n\in\Z}x\cdot t^ny
 =\varepsilon\bigl(\lambda_Y(x,y)\bigr)
\]
by \eqref{eq:lambda-def}, which is the assertion about the forms.
\end{proof}

\begin{corollary}\label{cor:lambda-Y}
The equivariant intersection form of $Y$ is $\lambda_Y\cong \langle 1\rangle $ over
$\Lambda$.
\end{corollary}

\begin{proof}
By \cref{thm:FQ} the form $\lambda_Y$ is nonsingular Hermitian, and
$H_2(Y;\Lambda)$ is free of rank one by \cref{prop:rank-one}; so
$\lambda_Y=\langle u\rangle$ with $u=\pm1$ by
\cref{lem:rank-one-forms}.  Augmenting,
$\langle\varepsilon(u)\rangle=\langle u\rangle$ is the ordinary
intersection form $Q_Y=\langle 1\rangle $ by
\cref{lem:augmentation} and Proposition~\ref{prop:intformZ},
implying $u=1$.
\end{proof}

\begin{proof}[Proof of \cref{thm:main}]
Part~(a) was verified at the end of Section~\ref{sec:coequalizer}.
For part~(b), $Y$ is a closed oriented smooth four-manifold with
$\pi_1(Y)\cong\Z$, and $\lambda_Y\cong\langle1\rangle$ by
\cref{cor:lambda-Y}, so \cref{thm:FQ} gives a homeomorphism
$Y\cong\CP^2\#(S^1\times S^3)$, consistent with
\eqref{eq:lambda-model}.  The remaining two reflections are conjugate
to $c_s$ by powers of $\sigma$ (\cref{cor:BY-comparison}), so their
orbit spaces are diffeomorphic to this one.
\end{proof}
  
\begin{proof}[Proof of \cref{cor:branched-cover}]
Each component of the branch locus $q_X(F)$ is the image of the fixed-point
set of an antiholomorphic isometric involution of $\B$, which is
therefore totally geodesic and hyperbolic \cite[Section~2.5]{GKL}, see
also \cite[Theorem~5.2]{KharKul}. Since $\PiCS$ is torsion free, that
component is a closed hyperbolic surface and therefore has $\chi<0$.
As $\chi(F)=-1$ by \cref{prop:chi-F}, the surface $F$ is connected,
and since it is a closed surface with $\chi=-1$, it is
$\#_3\RP^2$.
\end{proof}

\newpage

\appendix
\section{A verification in \texorpdfstring{\textsf{GAP}}{GAP}}\label{app:gap}

Both steps behind \cref{prop:Q-Z}---presenting $\PiCS$ and simplifying
$Q$---are implemented in the standard computational group theory system
\textsf{GAP} \cite{GAP4}:

\begin{lstlisting}
f := FreeGroup( "u", "v", "b" );;  u := f.1;;  v := f.2;;  b := f.3;;
G := f / [ u^3, v^4, b^3, (u*v)^2*((v*u)^2)^-1, v*b*v^-1*b^-1,
           (b*u*v)^3, (b*u*v*u)^2*v ];;
u := G.1;;  v := G.2;;  b := G.3;;  j := (u*v)^2;;
a1 := v*u*v^-1*j^4*b*u*v*j^2;;
a2 := v^2*u*b*u*v^-1*u*v^2*j;;
a3 := u^-1*v^2*u*j^9*b*v^-1*u*v^-1*j^8;;
H  := Subgroup( G, [ a1, a2, a3 ] );;
Index( G, H );                                # 864
# a presentation of Pi whose i-th generator IS a_i:
iso := IsomorphismFpGroupByGenerators( H, [ a1, a2, a3 ] );;
Pi  := Image( iso );;   A := GeneratorsOfGroup( Pi );;
List( [a1,a2,a3], g -> Image( iso, g ) ) = A; # true
AbelianInvariants( Pi );                      # [ 0, 0 ]
# adjoin the three coequalizer relators phi(a_i) a_i^-1:
Q := Pi / [ A[1]^-1*A[2]^-1*A[1]^-1*A[3]^-3*A[2]^3*A[1]^-1,
            A[1]^-1*A[2]^-1*A[1]*A[2]^2*A[1]^-1*A[2]^-1*A[1]*A[3]^-1
                *A[1]^-1*A[2]*A[1]*A[2]^-1,
            A[3]^-2 ];;
B := GeneratorsOfGroup( Q );;
IsAbelian( Q );                               # true
AbelianInvariants( Q );                       # [ 0 ]   so Q = Z
h := IsomorphismSimplifiedFpGroup( Q );;
RelatorsOfFpGroup( Image( h ) );
                       # [ F2*F1*F2^-1*F1^-1, (F2*F1^-1)^2*F1^-1 ]
List( B, x -> Image( h, x ) );                # [ F1, F2, <identity ...> ]
Index( Q, Subgroup( Q, [ B[1]^-1*B[2] ] ) );  # 1,  so Q = <t>, t = a1^-1 a2
[ (B[1]^-1*B[2])^2 = B[1], (B[1]^-1*B[2])^3 = B[2] ];    # [ true, true ]
\end{lstlisting}

Running the file produces 
\begin{lstlisting}
864
true
[ 0, 0 ]
true
[ 0 ]
[ F2*F1*F2^-1*F1^-1, (F2*F1^-1)^2*F1^-1 ]
[ F1, F2, <identity ...> ]
1
[ true, true ]
\end{lstlisting}


\begin{thebibliography}{99}

\bibitem{Armstrong}
M. Armstrong,
\emph{The fundamental group of the orbit space of a discontinuous group},
Proc. Cambridge Philos. Soc. \textbf{64} (1968), 299--301.

\bibitem{BY}
L. Borisov and S.-K. Yeung,
\emph{Explicit equations of the Cartwright--Steger surface},
\'Epijournal de G\'eom\'etrie Alg\'ebrique \textbf{4} (2020), Article 10.

\bibitem{CKY}
D. Cartwright, V. Koziarz, and S.-K. Yeung,
\emph{On the Cartwright--Steger surface},
J. Algebraic Geom. \textbf{26} (2017), 655--689.

\bibitem{CS}
D. Cartwright and T. Steger,
\emph{Enumeration of the $50$ fake projective planes},
C. R. Math. Acad. Sci. Paris \textbf{348} (2010), 11--13.

\bibitem{CS2}
D. Cartwright and T. Steger,
\emph{Finding generators and relations for groups acting on the hyperbolic
ball}, preprint; see \cite{CKY} and the files linked there.

\bibitem{Do}
S. Donaldson,
\emph{An application of gauge theory to four-dimensional topology},
J. Differential Geom. \textbf{18} (1983), 279--315.

\bibitem{Ehresmann}
C.~Ehresmann,
\emph{Les connexions infinit\'esimales dans un espace fibr\'e diff\'erentiable},
Colloque de topologie (espaces fibr\'es), Bruxelles 1950,
Georges Thone, Li\`ege; Masson et Cie., Paris, 1951, pp.~29--55.

\bibitem{Fr}
M.~H. Freedman,
\emph{The topology of four-dimensional manifolds},
J. Differential Geom. \textbf{17} (1982), 357--453.

\bibitem{FQ}
M.~H. Freedman and F.~Quinn,
\emph{Topology of 4-Manifolds},
Princeton Mathematical Series \textbf{39}, Princeton University Press, 1990.

\bibitem{GAP4}
The GAP~Group,
\emph{GAP -- Groups, Algorithms, and Programming, Version 4.12.1}, 2022,
\url{https://www.gap-system.org}.

\bibitem{GKL}
W.~Goldman, M.~Kapovich and B.~Leeb,
\emph{Complex hyperbolic manifolds homotopy equivalent to a
              {R}iemann surface},
Comm. Anal. Geom. \textbf{9} (2001), 61--95.

\bibitem{HambletonSurvey}
I.~Hambleton,
\emph{Intersection forms, fundamental groups and 4-manifolds},
Proceedings of the G\"okova Geometry--Topology Conference 2008,
G\"okova Geometry/Topology Conference (GGT), 2009, 137--150.

\bibitem{Hatcher}
A.~Hatcher,
\emph{Algebraic Topology},
Cambridge University Press, 2002.

\bibitem{Hirzebruch}
F.~Hirzebruch, 
\emph{The signature of ramified coverings},
 in: Global {A}nalysis ({P}apers in {H}onor of {K}. {K}odaira)
  (1969), 253--265.
 
\bibitem{JO}
K.~J\"{a}nich and E.~Ossa,
\emph{On the signature of an involution},
Topology \textbf{8} (1969), 27--30.

\bibitem{KharKul}
V.~Kulikov and V.~Kharlamov, 
\emph{On real structures on rigid surfaces},
Izv. Ross. Akad. Nauk Ser. Mat. \textbf{66} (2002), 133--152.

\bibitem{KY}
V.~Koziarz and S.-K. Yeung,
\emph{Stability of the Albanese fibration on the Cartwright--Steger surface},
Taiwanese J. Math. \textbf{25} (2021), 251--256.

\bibitem{Kuiper}
N.~Kuiper, 
\emph{The quotient space of {${\bf C}P(2)$} by complex conjugation
              is the {$4$}-sphere},
Math. Ann. \textbf{208} (1974), 175--177.

\bibitem{Massey}
W.~Massey, 
\emph{The quotient space of the complex projective plane under
              conjugation is a {$4$}-sphere},
Geometriae Dedicata \textbf{2} (1973), 371--374.

\bibitem{Miyaoka}
Y.~Miyaoka,
\emph{On the Chern numbers of surfaces of general type},
Invent. Math. \textbf{42} (1977), 225--237.

\bibitem{PY}
G.~Prasad and S.-K. Yeung,
\emph{Fake projective planes},
Invent. Math. \textbf{168} (2007), 321--370.
Addendum, Invent. Math. \textbf{182} (2010), 213--227.

\bibitem{Rito}
C.~Rito,
\emph{Surfaces with canonical map of maximum degree},
J. Algebraic Geom. \textbf{31} (2022), 127--135.

\bibitem{SSz}
  A. Stipsicz and Z. Szab\'o,
\emph{On the smooth structure of the real quotient of the Cartwright-Steger surface}, in preparation.
 
\bibitem{Yau}
S.-T. Yau,
\emph{Calabi's conjecture and some new results in algebraic geometry},
Proc. Nat. Acad. Sci. U.S.A. \textbf{74} (1977), 1798--1799.
\end{thebibliography}
\end{document}